\documentclass[12pt]{article}
\usepackage{defs2}
\usepackage[normalem]{ulem}
\usepackage{amsthm}

\title{Multiplicative differential algebraic $K$-theory and applications}
\author{Ulrich Bunke\thanks{{Fakult\"at f\"ur Mathematik},
Universit{\"a}t Regensburg,
93040 Regensburg,
GERMANY, ulrich.bunke@mathematik.uni-regensburg.de}
~and
Georg Tamme\thanks{Fakult\"at f\"ur Mathematik, Universit\"at Regensburg, 93040 Regensburg, GERMANY, georg.tamme@mathematik.uni-regensburg.de}
}
\date{October 29, 2015}

\newcommand{\hbK}{{\widehat{\mathbf{K}}}}

\newcommand{\oben}[1]{\!\phantom{.}^{#1}}

\newcommand{\CAlg}{\mathbf{CAlg}}

 \newcommand{\hPic}{{\widehat{\tt Pic}}}
\newcommand{\bs}{{\mathbf{s}}}
\newcommand{\can}{{\tt can}}
\newcommand{\Gal}{{\tt Gal}}

\newcommand{\IDR}{{\mathbf{IDR}}}

\newcommand{\SemiRing}{{\mathbf{Rig}}}
\newcommand{\Ring}{{\mathbf{Ring}}}

\newcommand{\Spec}{\mathtt{Spec}}

\newcommand{\beil}{{\tt r^{\mathrm Beil}}}
\newcommand{\cI}{{\mathcal{I}}}

\newcommand{\GrRings}{\mathbf{GrRings}}

\newcommand{\bbA}{{\mathbb{A}}}

\newcommand{\cL}{{\mathcal{L}}}
\newcommand{\cW}{{\mathcal{W}}}

\newcommand{\hcycl}{\widehat{\tt cycl}}

\newcommand{\bK}{{\mathbf{K}}}

\newcommand{\Reg}{{\mathbf{ Reg}}}

\newcommand{\cO}{{\mathcal{O}}}
\newcommand{\cU}{{\mathcal{U}}}

\newcommand{\cD}{{\mathcal{D}}}

 \newcommand{\Cone}{{\tt Cone}}
 \newcommand{\DR}{\mathbf{DR}}
 \newcommand{\Vect}{{\tt Vect}}
 \newcommand{\CommGroup}{{\mathbf{CommGroup}}}
 \newcommand{\CommMon}{{\mathbf{CommMon}}}
 \newcommand{\Cat}{{\mathbf{Cat}}}

\DeclareMathOperator{\hdeg}{\widehat{deg}}
\DeclareMathOperator{\des}{des}

\renewcommand{\Re}{\operatorname{Re}}
\renewcommand{\Im}{\operatorname{Im}}

\numberwithin{equation}{section}

\newtheorem{theorem}[equation]{Theorem}
\newtheorem{prop}[equation]{Proposition}
\newtheorem{lem}[equation]{Lemma}
\newtheorem{ddd}[equation]{Definition}
\newtheorem{kor}[equation]{Corollary}

\theoremstyle{remark}
\newtheorem{rem}[equation]{Remark}
\newtheorem{ex}[equation]{Example}

\begin{document}

\maketitle
\begin{abstract}
We construct a version of Beilinson's regulator as a map of sheaves of commutative ring
spectra and use it to define a multiplicative variant of differential algebraic $K$-theory. We
use this theory to give an interpretation of Bloch's construction of $K_3$-classes and the relation
with dilogarithms. Furthermore, we provide a relation to Arakelov theory via the arithmetic
degree of metrized line bundles, and we give a proof of the formality of the algebraic $K$-theory
of number rings.
\end{abstract}

\tableofcontents

\section{Introduction}


Let $X$ be an arithmetic scheme, i.e.~a regular separated scheme of finite type over the integers. Its algebraic $K$-theory $K_{*}(X)$ is an object of fundamental interest in arithmetic. 
 The algebraic $K$-theory of $X$  is connected with  the  absolute Hodge cohomology $H^{*}_{\cA\cH}(X,\R(\bullet))$ via a Chern character map, called the Beilinson regulator
\[
K_{i}(X) \to H^{2p-i}_{\cA\cH}(X,\R(p))\ , \quad p,i\geq 0\ .
\]
An important, but extremely difficult problem is to construct $K$-theory classes and to compute their images under the regulator map.

The papers \cite{2013arXiv1306.0247B, buta} initiated a new approach to this problem. The idea is to represent algebraic $K$-theory classes of $X$ by bundles on $M\times X$ for smooth manifolds $M$.  In greater detail this goes as follows. The $K$-groups of $X$ are the homotopy groups of the algebraic $K$-theory spectrum $\bK(X)$.  This spectrum defines a cohomology theory $\bK(X)^{*}$ on topological spaces so that, e.g., $\bK(X)^{0}(S^{n})\cong K_{0}(X)\oplus K_{n}(X)$. The cohomology theory $\bK(X)^{*}$ admits a differential refinement    denoted by $\hbK\oben{*}(M\times X)$. This differential algebraic $K$-theory is a functor of two variables, a smooth manifold $M$ and a scheme $X$ as above.  A class $\hat x\in \hbK\oben{*}(M\times X)$ combines the information of a class $x\in \bK(X)^{*}(M)$ and a   differential form on the manifold $M\times X(\C)$ {representing} the image of $x$ under Beilinson's regulator with secondary data. Thus, if we know a differential refinement $\hat x$ of $x$ then, 
philosophically, it is easy to calculate 
the Beilinson regulator of $x$.


The tool to construct differential algebraic $K$-theory classes  {is the cycle map}.  {It produces such classes from bundles on $M\times X$ equipped with additional geometric data.}
Here a  bundle on $M\times X$ is a  {vector bundle on the ringed space $(M\times X, \pr_X^{-1}\cO_X)$. The geometric extra structure is a hermitian metric and a connection on the associated complex vector bundle on $M\times X(\C)$.}
The differential form representing the Beilinson regulator  {of the corresponding $K$-theory class} is obtained using standard Chern-Weil  theory.

The aim of the present paper is to develop a multiplicative version of differential algebraic
$K$-theory and to illustrate it in some applications. The  cup product allows us to construct new classes from given ones, but more interestingly, we will  employ the  secondary information
captured by the differential algebraic $K$-theory in an essential way.

In order to achieve this goal we need a version of Beilinson's regulator on the level of ring spectra. Here our result is not completely satisfactory, as we have to replace absolute Hodge cohomology by the weaker analytic Deligne cohomology, which coincides with the former only for proper schemes.
We construct a sheaf of ring spectra $\bK$ on the site consisting of products of a smooth manifold and an arithmetic scheme such that $\pi_*(\bK(M\times X)) \cong \bK(X)^{-*}(M)$. To this end we apply a    suitable group completion machine to the category of vector bundles on the ringed space $(M\times X, \pr_X^{-1}\cO_X)$. We furthermore construct a sheaf of differential graded algebras $\IDR$ which computes analytic Deligne cohomology and use characteristic forms on vector bundles on the manifolds $M\times X(\C)$ to construct a map of sheaves of ring spectra ($H$ denotes the Eilenberg-MacLane functor)
\[
\beil\colon \bK \to H(\IDR)
\]
which on homotopy groups agrees with the Beilinson regulator. This is the main new contribution of the paper.

Once the multiplicative Beilinson regulator is established,  we introduce 
the multiplicative differential algebraic $K$-theory and a multiplicative version of the cycle map in Section \ref{juni1010}.

The remainder of the present paper is devoted to applications and illustrating how classical
constructions from arithmetic fit into the framework of differential algebraic $K$-theory.

In Section \ref{juni1002} we use  multiplicative differential algebraic $K$-theory in order to
construct a secondary invariant from the Steinberg relation. As an application we give a
conceptual explanation of Bloch's construction of elements in $K_{3}$ of a  number ring
from cycles in the Bloch
complex, whose images under the regulator
map can be described explicitly in terms of the dilogarithm function.

In Arakelov theory one studies metrized line bundles on number rings and their arithmetic degree. We explain in Section \ref{juni1003} how this can be understood entirely in  the framework of differential algebraic $K$-theory.
 
Finally, in Section \ref{jun1004} we show that the real homotopy type of the algebraic $K$-theory  spectrum $\bK(\Spec(R))$ of rings of integers $R$ in number fields is modelled by the commutative algebra $K_{*}(R)$ in a way which is natural in $R$.  The precise formulation of this result is  Theorem \ref{juni1001}  and  uses the notion of formality introduced in Definition \ref{fkwelfwefwefewfewfefe}. 

 \bigskip

 { 
{\it Acknowledgements.} We would like to thank the referee and Chuck Weibel for several comments. They helped to improve the exposition of the paper a lot.
}

\section{Multiplicative theory}

%

In this section we define algebraic $K$-theory as a sheaf $\bK$ of commutative ring spectra on a
  site of products of a smooth manifold and a regular scheme (see Section \ref{sec:sites}
below). We furthermore define a sheaf of differential graded algebras $\IDR$ which calculates the
analytic Deligne cohomology (Section \ref{sec:IDR}).

The main result is the construction of a version of Beilinson's regulator with values
in analytic Deligne cohomology as a map between sheaves of ring spectra
\[
\beil\colon \bK \to H(\IDR)
\]
where $H(\IDR)$ is the Eilenberg-MacLane spectrum associated to $\IDR$ (Theorem \ref{may2760})
using multiplicative characteristic forms (Section~\ref{sec:geometries}).

Throughout the paper we use the language of $(\infty,1)$-categories as developed by Lurie
\cite{MR2522659} and simply called $\infty$-categories in the following. We view an ordinary
category as an $\infty$-category by taking its nerve.

\subsection{The sites}
\label{sec:sites}

We let $\Mf$ denote the category of smooth manifolds with the open covering topology. 
Here a smooth manifold is a smooth manifold with corners locally modeled on $[0,\infty)^{n}\subset \R^{n}$, $n\in \nat$. The category $\Mf$  contains manifolds with boundary and  is closed under products. $\Mf$ in particular  contains the interval $I=\Delta^{1}=[0,1]$ and the standard simplices $\Delta^{p}$ for all $p\in \nat$.
We let $\Reg_{\Z}$ denote the category of regular separated schemes of finite type over $\Spec(\Z)$ with the topology of Zariski open coverings. 
Manifolds and schemes are combined in the product
$\Mf\times \Reg_{\Z}$  of  these sites.

Let $\bC$  be a presentable $\infty$-category \cite[Ch.~5]{MR2522659}. 
We can consider the $\infty$-category of functors $ \Fun((\Mf\times \Reg_{\Z})^{op},\bC)$. Objects
in this $\infty$-category will also be referred to as \emph{presheaves}.
\begin{ddd}
An object $F\in \Fun((\Mf\times \Reg_{\Z})^{op},\bC)$ satisfies descent if
$F$ sends disjoint unions to products and 
for every covering $\cU$ of an object   $M\times X\in \Mf\times \Reg_{\Z}$  the natural map 
$$F(M{\times X}) \to \lim_{\Delta^{op}} F(\cU_{{\bullet}})$$
is an equivalence, where $\cU_{{\bullet}}\in (\Mf\times \Reg_{\Z})^{\Delta^{{op}}}$ denotes the \v{C}ech nerve of $\cU$.
\end{ddd}
We write
$\Fun^{desc}((\Mf\times \Reg_{\Z})^{op},\bC)$
for the full subcategory of objects satisfying descent. These objects will be called \emph{sheaves}.
 The inclusion as a full subcategory admits a left adjoint $L$ called
sheafification \cite[Lemma~6.2.2.7]{MR2522659}. We express this by the diagram
\[
L\colon \Fun((\Mf\times \Reg_{\Z})^{op},\bC)\leftrightarrows \Fun^{desc}((\Mf\times
\Reg_{\Z})^{op},\bC).
\]

We will also need the notion of homotopy invariance (in the manifold direction), which should not be
confused with $\bbA^{1}$-homotopy invariance in the algebraic direction.
Let $I:=[0,1]$ be the unit interval.

\begin{ddd}
An object
$F\in \Fun((\Mf\times \Reg_{\Z})^{op},\bC)$ is homotopy invariant (in the manifold direction),
if the natural map
\[
\pr^{*}\colon F(M\times X)\to  F(I\times M\times X)
\]
is an equivalence for every object   $M\times X\in \Mf\times \Reg_{\Z}$.
\end{ddd}
We write
$\Fun^{{I}}((\Mf\times \Reg_{\Z})^{op},\bC)$ for the full subcategory of homotopy invariant
objects. We again have an adjunction
\[
\cH^{pre}\colon \Fun((\Mf\times \Reg_{\Z})^{op},\bC) \leftrightarrows \Fun^{{I}}((\Mf\times \Reg_{\Z})^{op},\bC)
\]
and $\cH^{pre}$ is called the \emph{homotopification}.
We denote by $\Fun^{desc,I}((\Mf\times \Reg_{\Z})^{op},\bC)$ the full subcategory of
presheaves satisfying both, homotopy invariance and descent. Then we have a commutative diagram in
$\infty$-categories
$$
\xymatrix{
\Fun^{desc,I}((\Mf\times \Reg_{\Z})^{op},\bC) \ar[d]\ar[r]&\Fun^{desc}((\Mf\times
\Reg_{\Z})^{op},\bC)\ar[d]\\\Fun^{I}((\Mf\times \Reg_{\Z})^{op},\bC)\ar[r]&\Fun((\Mf\times
\Reg_{\Z})^{op},\bC)
}
$$
where all morphisms are inclusions of full subcategories. Taking adjoints we get a commutative diagram of localizations, 
\begin{equation}
\begin{split}
\label{gshjgshjgjsgjwqsq}
\xymatrix{
\Fun^{desc,I}((\Mf\times \Reg_{\Z})^{op},\bC)  &\Fun^{desc}((\Mf\times
\Reg_{\Z})^{op},\bC)\ar[l]^{\cH} \\
\Fun^{I}((\Mf\times\Reg_{\Z})^{op},\bC)\ar[u]^{L_{{I}}} & \Fun((\Mf\times \Reg_{\Z})^{op},\bC) \ . \ar[u]^{L}\ar[l]^{\cH^{pre}}
}
\end{split}
\end{equation}
In order to see that the horizontal adjunctions  exists one can use  identifications of the form
$$ \Fun((\Mf\times \Reg_{\Z})^{op},\bC)\simeq  \Fun(\Mf^{op},\Fun(\Reg_{\Z}^{op},\bC))$$
and refer to \cite[Section~2]{bnv}. Then diagram \eqref{gshjgshjgjsgjwqsq} shows
 that sheafification commutes with homotopification  in the sense that $L_I\circ
\cH^{pre} \simeq \cH \circ L$. 
 Here $L_{I}$ and $\cH$ are the sheafification and the homotopification functors on the
respective subcategories. It is not   clear that $\cH$ is the restriction of
$\cH^{pre}$. Again, we refer  to \cite[Section 2]{bnv} for more details.

Note that any functor $\Phi:\bC\to \bC^{\prime}$ between presentable $\infty$-categories induces a functor
${\Phi_*\colon} \Fun((\Mf\times \Reg_{\Z})^{op},\bC)\to\Fun((\Mf\times \Reg_{\Z})^{op},\bC^{\prime})$ which
preserves homotopy invariant objects. In contrast, $\Phi_*$  preserves sheaves in general only if $\Phi$ commutes with limits. 
We will usually write $\Phi$ for $\Phi_{*}$ in order simplify the notation.

 {Later, we will need the following explicit description of the homotopification.}
We first define a functor
$$\bs:\Fun((\Mf\times \Reg_\Z)^{op},\bC)\to \Fun((\Mf\times \Reg_\Z)^{op},\Fun(\Delta^{op},\bC))$$
as the adjoint of
$$(\Mf\times \Reg_\Z)^{op} \times \Delta^{op}\to   (\Mf\times \Reg_\Z)^{op}\ , (M\times X\times [p])\mapsto \Delta^{p}\times M\times X\ ,$$
where $\Delta^{p}\in \Mf$ denotes the $p$-dimensional standard simplex.
We further set
\begin{equation}\label{apr1902}
\bar \bs:=\colim_{\Delta^{op}} \circ \bs:\Fun((\Mf\times \Reg_\Z)^{op},\bC)\to\Fun((\Mf\times \Reg_\Z)^{op},\bC)\ .
\end{equation}

\begin{lem}
\begin{enumerate}
\item There is a natural map $\id\to \bar  \bs$.
\item If $X\in {\Fun  ((\Mf\times \Reg_\Z)^{op},\bC)}$ is homotopy invariant, then the natural map $X\to \bar \bs(X)$ is an equivalence.
\item If $f$ is a morphism in ${\Fun((\Mf\times \Reg_\Z)^{op},\bC)}$  such that $\bar \bs(f)$ is an equivalence, then $\cH^{pre}(f)$ is an equivalence.
\item The map $\id\to \bar \bs$ is equivalent to the unit of the homotopification $\id\to \cH^{pre}$ on
$\Fun((\Mf\times \Reg_{\Z})^{op},\bC)$.
\end{enumerate}
\end{lem}
\begin{proof} 
The last statement implies the first three, which are exercises.  Details can be found in
\cite[{4.29}]{skript}. For (4) we refer to \cite[Lemma 7.5]{bnv}.
\end{proof}

\subsection{The multiplicative Deligne complex}
\label{sec:IDR}

We consider the site of smooth complex varieties $\Sm_{\C}$ with the Zariski topology  and the
product
$\Mf\times \Sm_{\C}$.
We denote by $\Ch$ the  1-category of complexes of abelian groups considered as
$\infty$-category
and by $\Ch[W^{-1}]$ its localization with quasi-isomorphisms inverted.  
We have  the sheaf of complexes $A\in  \Fun^{desc}((\Mf\times \Sm_{\C})^{op},\Ch)$ of complex
valued  {smooth}
differential forms.  It contains the
subsheaf of complexes of real valued forms $A_{\R}$. Obviously, $A \cong A_{\R}\otimes_{\R} \C$. 
The sheaf of complexes $A$ furthermore has a decreasing Hodge filtration $\cF$
such that elements in $\cF^pA(M\times X)$ are locally of the form
\[
\sum_{I,J,K, |J|\geq p} \omega_{I,J,K}\: dx^I \wedge dz^J \wedge d\bar z^K
\]
where the $z_j$'s are local holomorphic coordinates on $X$ and the $x_i$'s are local coordinates on $M$
(in contrast to \cite[Section 4.2]{buta}, we forget the
$\log$-condition and the weight filtration). 
Since, degree-wise, these sheaves of complexes 
consist of modules over the sheaf of smooth functions, they satisfy descent, i.e.~they
are sheaves when considered as objects in $\Fun((\Mf\times
\Sm_{\C})^{op},\Ch[W^{-1}])$ (see \cite[Lemma 7.12]{bnv} for an argument). 
By the Poincar\'e Lemma they are also homotopy invariant.

We let
$B\colon \Reg_{\Z}\to \Sm_{\C}$ {be the functor mapping a} scheme $X$ to the smooth complex variety $X\times_\Z \C$. 
Then $(\id\times B)^{*} A\in {\Fun^{desc}((\Mf\times\Reg_{\Z})^{op},\Ch)}$ has a $\Gal(\C/\R)$-action which preserves the  Hodge
filtration. 
The sheaf of complexes $\DR (p)\in  \Fun^{desc}((\Mf\times \Reg_{\Z})^{op},\Ch)$  is
defined by
$$
\DR (p):=[(\id\times B)^{*}\DR_\C (p)]^{\Gal(\C/\R)}\ ,
$$ 
where  
\[
\DR_\C (p):= \Cone \left(  (2\pi i)^{p}   A_{\R}\oplus   \cF^{p} A \xrightarrow{\alpha\oplus\beta\mapsto \alpha-\beta}  A \right)[2p-1] \ .
\]
 Here $(\,.\,)^{\Gal(\C/\R)}$ denotes the object-wise fixed points under the group $\Gal(\C/\R)$. 
Note that all sheaves that appear above have in fact values in complexes of real vector spaces.   Furthermore, taking
invariants under the finite group  $\Gal(\C/\R)$ is an exact functor on real vector spaces with $\Gal(\C/\R)$-action. Therefore, taking  $\Gal(\C/\R)$-invariants
  preserves the descent and homotopy invariance conditions. Consequently, 
we can consider $\DR(p)\in  \Fun^{desc,I}((\Mf\times \Reg_{\Z})^{op},\Ch{[W^{-1}]})$
\begin{rem}\label{aug0601}
For a smooth complex variety $X$,
the complex $\DR_{\C}(p)(X)$ calculates the analytic Deligne cohomology $H^{*}_{\cD,\mathrm{an}}(X,
\R(p))$ up to a shift of degrees by $2p$. If, in the definition of the cone, one replaces the
complexes of smooth forms $A_{\R}, A$ by their $\log$-versions $A_{\R,\log}, A_{\log}$ (consisting
of forms which extend  to some compactification of $X$ with logarithmic poles along the boundary of
$X$, see~\cite[Section 4.2]{buta}) one obtains the so-called Beilinson-Deligne or weak absolute
Hodge cohomology $H^{*}_{B\cD}(X,\R(p))$. There is a natural map $H^{*}_{B\cD}(X,\R(p)) \to
H^{*}_{\cD,\mathrm{an}}(X, \R(p))$ which, in general, is neither injective nor surjective. It is an
isomorphism if $X$ is also proper over $\C$. If one moreover introduces the weight filtration
$\hat\cW$ and replaces $A_{\R,\log}, A_{\log}$ by the subcomplexes $\hat\cW_{2p}A_{\R,\log},
\hat\cW_{2p}A_{\log}$, one obtains the absolute Hodge cohomology $H^{*}_{\cA\cH}(X,\R(p))$
introduced by Beilinson \cite{BeilinsonHodgeCohom}. This is 
the cohomology theory used in \cite{buta}. It follows from Deligne's theory of weights that the natural map $H^{*}_{\cA\cH}(X,\R(p)) \to H^{*}_{B\cD}(X,\R(p))$ is an isomorphism in degrees $* \leq p$, and in degrees $*\leq 2p$ if $X$ is proper. 
\end{rem}

 In the following, we define  a sheaf
$\IDR(p)\in \Fun^{desc}((\Mf\times \Reg_{\Z})^{op},\Ch)$ which is  object-wise quasi-isomorphic  to $\DR(p)$, and  which is better behaved with respect to the multiplicative structures. 
We define the morphism
$$\cI:\Mf\to \Mf\ , \quad M\mapsto [0,1]\times M\ .$$  It induces a corresponding morphism
$\cI\times \id_{\Sm_\C}:\Mf\times \Sm_\C\to \Mf\times \Sm_\C$.
For a presheaf $\cF$ on $\Mf\times \Sm_\C$ we define
$\cI\cF:=(\cI\times \id_{\Sm_\C})^{*}\cF$.
\begin{ddd}\label{jun1302}
We define  
$$\IDR_{\C}(p)\subseteq \cI  A [2p]$$ 
to be the subsheaf with values in $\Ch$ determined by the conditions that
$\omega\in \IDR_{\C}(p)(M\times X)$ iff
\begin{enumerate}
\item $\omega_{|\{0\}\times M {\times X}}\in   (2\pi i)^{p}    A_{\R}(M\times X)[2p]$
\item $\omega_{|\{1\}\times M {\times X}}\in  \cF^{p} A(M\times X)[2p]$\ .
\end{enumerate}
We set $\IDR_{\C}:=\prod_{p\ge 0} \IDR_{\C}(p)$
 and  define
$$\IDR :=[(\id\times B)^{*} \IDR_\C]^{\Gal(\C/\R)}\ .$$
\end{ddd} 
{An algebraic analog of this complex was used by Burgos and Wang \cite{BurgosWang}.}
\begin{prop}
There is an object-wise quasi-isomorphism
\begin{equation}\label{may2350}
q:\IDR  (p)\to \DR  (p).
\end{equation}
\end{prop}
\begin{proof}
We define a morphism  of sheaves of complexes 
\begin{equation}\label{may2350nnn}
q_\C\colon \IDR_\C (p)\to \DR_\C (p)
\end{equation} 
as follows.
A form $\omega \in \IDR_\C (p)(M)$ gives rise to forms
\begin{enumerate}
\item
$\omega_\R:=\omega_{|\{0\}\times M {\times X}}\in  (2\pi i)^{p}     A_{\R}(M\times X)[2p]$
\item 
$\omega_\cF:=\omega_{|\{1\}\times M {\times X}}\in \cF^{p}    A(M\times X)[2p]$
\item
$\tilde \omega:=\int_{[0,1]\times M{\times X}/M{\times X}}\omega \in   A(M\times
X)[2p-1]$.
\end{enumerate}
We define 
$$
q_\C (\omega):=(\omega_\R\oplus \omega_\cF,{-}\tilde \omega)\in \DR_{\C}(M\times X)\ .
$$ 
We have 
$$dq_\C(\omega)=d(\omega_\R\oplus \omega_\cF,-\tilde \omega)=(d\omega_\R\oplus d\omega_\cF,d\tilde\omega+\omega_\R-\omega_\cF) $$
and 
\begin{gather*}
q_\C(d\omega) = (d\omega_\R\oplus d\omega_\cF, -\int_{[0,1]\times M/M} d\omega) = \\
= (d\omega_\R\oplus d\omega_\cF, d \int_{[0,1]\times M/M} \omega +\omega_\R - \omega_\cF) = (d\omega_\R\oplus d\omega_\cF, d\tilde\omega +\omega_\R - \omega_\cF),
\end{gather*}
a calculation using Stokes' theorem. Hence $q_{\C}$ is a map of complexes.
\begin{lem}\label{may2320}
For every $p\ge 0$ 
the map $q_\C :\IDR_{\C}(p)\to \DR_{\C}(p)$  is an object-wise quasi-isomorphism.
\end{lem} 
\begin{proof}
We abbreviate
$$S:= A /(2\pi i)^p A_{\R}[2p]\ , \quad 
T:= A/ \cF^{p} A_{\R}[2p]\ .$$
Then we have an exact sequence
\begin{equation}\label{may2310}0\to \IDR (p) \to \cI A[2p] \to S\oplus T\to 0\ ,\end{equation}
where the first map is the inclusion and the second is given by the evaluation at the end points of the interval.
We further have a natural exact sequence
\begin{equation}\label{may2311}0\to \DR (p) \to \Cone(A\oplus A\to A)[2p-1]\to \Cone(S\oplus T\to 0)[-1]\to 0\ .\end{equation}
We define a map of exact sequences $\eqref{may2310} \to \eqref{may2311}$
using the map $q_\C $ in the first entry, the same formula as for $q_{{\C}}$ in the second, and 
the obvious identity map at the last entry. Since the interval
$[0,1]$ is contractible it follows from the relative Poincar\'e Lemma that the middle map is a quasi-isomorphism. Since the last map is an isomorphism, it follows from the Five Lemma that $q_\C $ is a quasi-isomorphism, too. 
\end{proof}

We observe that $(\id\times B)^{*}q_\C $   commutes with the $\Gal(\C/\R)$-action and therefore induces an equivalence  $q  :\IDR(p)\to \DR(p)$, too. This finishes the proof of the Proposition.
\end{proof}

It follows from
Lemma \ref{may2320} and the sheaf and homotopy invariance properties of $\DR $ that we can consider
$$\IDR \in \Fun^{desc,I}((\Mf\times \Reg_{\Z})^{op},\Ch[W^{-1}])\ .$$

We now observe that the filtration  $\cF$ as well as the real subspaces  are compatible with the multiplication $$\wedge : A \otimes A  \to A \ .$$ We therefore get products
$$\wedge :\IDR (p)\times \IDR  (q)\to \IDR (p+q)\ .$$
Taking the product over all $p$, we get as final result:
\begin{kor}\label{jun1301}
The product
$$\IDR :=\prod_{p\ge 0}\IDR (p)$$
has the structure of a sheaf of bi-graded graded commutative $d$-algebras.
\end{kor} 
We denote the symmetric monoidal $\infty$-categories of chain complexes and chain complexes 
with quasi-isomorphisms inverted with the tensor product by $\Ch^{\otimes}$ and
$\Ch[W^{-1}]^{\otimes}$,  {respectively}. The notation for  commutative algebra objects is $\CAlg$. 
Commutative differential graded algebras are objects of $\CAlg(\Ch^{\otimes})$. They can be considered as  objects in $\CAlg(\Ch[W^{-1}]^{\otimes})$.
Since the forgetful
functor $\CAlg(\Ch[W^{-1}]^{\otimes}) \to \Ch[W^{-1}]^{\otimes}$ is a right adjoint, limits in
commutative algebras are computed on underlying objects. Consequently,  $\IDR$  can   naturally be  considered as an object
\begin{equation}\label{descentIDR}
\IDR \in \Fun^{desc,I}((\Mf\times \Reg_{\Z})^{op},\CAlg  (\Ch[W^{-1}]^{\otimes})).
\end{equation}

\subsection{Geometries  {and characteristic forms}}
\label{sec:geometries}

We first consider $M\times X\in \Mf\times \Sm_\C$. We view $M \times X$ as a locally ringed space
with structure sheaf $\cO_{M\times X} := \pr_X^{-1}\cO_X$ given by the inverse image of the sheaf
$\cO_X$ under the projection to $X$. A sheaf of finitely generated locally free ${\cO_{M\times
X}}$-modules   will be called a bundle on $M\times X$. 
If $V$ is a bundle on $M\times X$ we have an associated complex vector bundle on $M\times X(\C)$
which we abusively denote by the same symbol. It naturally carries a flat partial connection
$\nabla^I$ in the $M$-direction and a holomorphic structure $\bar\partial$ in the $X$-direction,
which is constant with respect to $\nabla^I$, i.e.~$[\nabla^I,\bar\partial]=0$.
\begin{ddd}[see~{\cite[Def.~4.12]{buta}}]
A geometry on the bundle $V$ is given by a pair $(h^V,\nabla^{II})$ consisting of a hermitian
metric $h^V$ on $V$ and a partial connection $\nabla^{II}$ in the $X$-direction that extends the
holomorphic structure $\bar\partial$.
\end{ddd}

We form the connection $\nabla:=\nabla^{I}+\nabla^{II}$ and let  $\nabla^{u}$ be its unitarization
with respect 
to $h^{V}$. In \cite{buta} we use these connections in order to define a characteristic form in $\DR (M\times X)$. In the present paper we adjust the notion of a geometry such that we obtain a lift of the characteristic form to $\IDR (M\times X)$, see Lemma \ref{jun1305}.

Let $\pr:I\times M\times X\to M\times X$ denote the projection.
 \begin{ddd}\label{jun1601}
An extended geometry {$g$} on $V$ is a triple 
${g=}(( h^{V},\nabla^{II}),{\widetilde\nabla})$  {consisting}
of a  geometry on  $V$ and a connection ${\widetilde\nabla}$ on
$\pr^{*}V$ 
such that
\begin{enumerate}
\item ${\widetilde\nabla}_{|\{0\}\times M\times X}=\nabla^{u}$
\item ${\widetilde\nabla}_{|\{1\}\times M\times X}=\nabla$\ .
\end{enumerate}
\end{ddd}
 
We now consider the arithmetic situation $M\times X\in \Mf\times \Reg_{\Z}$.
We keep calling a  sheaf of finitely generated locally free ${\cO_{M\times X}}$-modules a bundle. 
For the notion of $\Gal(\C/\R)$-invariance in the following definition we refer to \cite[Definition 4.31]{buta}.
 \begin{ddd}\label{jun1602}
 An extended  geometry {$g$} on a bundle
$V$ {on $M\times X\in \Mf\times \Reg_{\Z}$} is a {$\Gal(\C/\R)$}-invariant extended  geometry {$g$} on the  bundle $(\id\times B)^{*}(V)$.
\end{ddd}
Geometries  and extended geometries exist and can be glued with partitions of unity on $M$.
Compared with \cite{buta} the situation is simplified since we drop the condition of being good.
 Examples are given by the canonical extensions:
\begin{ddd}\label{jun2202} 
{Given a  geometry $(h^V,\nabla^{II})$ on the bundle $V$, we}
define the {associated} canonical extended geometry 
\begin{equation*}
\can(h^{V},\nabla^{{II}}):=((h^{V},\nabla^{{II}}),{\widetilde\nabla})
\end{equation*}
by taking for ${\widetilde\nabla}$ the linear path from $\nabla^{u}$ to $\nabla$.
\end{ddd}

For any $M\times X\in \Mf\times \Reg_\Z$ we denote the groupoid of bundles with extended geometry on
$M\times X$ {and isomorphisms respecting the extended geometry} by $i\Vect^{exge}(M\times X)$.

For a closed symmetric monoidal presentable $\infty$-category $\bC^\otimes$  we denote by
$\SemiRing(\bC^\otimes)$ the $\infty$-category of semiring objects in $\bC$
(see \cite[Def.~7.1]{2013arXiv1305.4550G}). The typical example of a semiring in $\Set^{\times}$
is the semiring of integers $\nat$.
 We let $\Cat[W^{-1}]^\times$ be the $\infty$-category of
categories with categorical equivalences inverted, equipped with its cartesian symmetric monoidal
structure.  {A semiring in $\Cat[W^{-1}]^\times$ will be called a Rig-category.}
Then  {a typical Rig-category} is the category of vector spaces over some field
 with the operations $\oplus$ and $\otimes$. This follows
  from the recognition principle \cite[Thm. 8.8]{2013arXiv1305.4550G}.  This principle implies that, using
direct sum and tensor product of bundles with geometry, we can consider $i\Vect^{exge}$ as a
sheaf of Rig-categories
\[
i\Vect^{exge}\in \Fun^{desc}((\Mf\times \Reg_{\Z})^{op},\SemiRing(\Cat[W^{-1}]^{\times}))\ . 
\] 
We furthermore interpret   $\pi_{0}(i\Vect^{exge})$ and $Z^{0}(\IDR)$ as presheaves of  semirings
$$\pi_{0}(i\Vect^{exge}), \ Z^{0}(\IDR)\in \Fun((\Mf\times \Reg_{\Z})^{op},\SemiRing(\Set^{\times}))\ .$$

We let $R^\nabla$ denote  the  curvature of a connection $\nabla$.
 Furthermore, by
\[
\ch_{2p}(\nabla) := \left[\Tr \exp(-R^\nabla)\right]_{2p} =(-1)^{p}\:\Tr (R^{\nabla})^{p}
\] 
we denote
the component of the unnormalized Chern character form in degree $2p$.
\begin{ddd}\label{may2730}
We define the  transformation of presheaves of semirings
$${\widetilde \omega}\colon\pi_{0}(i\Vect^{exge})\to  Z^{0}(\IDR)$$
by
$${\widetilde \omega}(V,g):=\prod_{p\ge 0} \ch_{2p}({\widetilde \nabla})\ .$$
\end{ddd}
A priori, 
$$
\prod_{p\ge 0}\ch_{2p}({\widetilde \nabla})\in \prod_{p\ge 0}\cI {A}(M\times B(X)) ,
$$ 
but the conditions for $\widetilde \nabla$ at the end-points of the interval immediately imply 
 that this product of forms  belongs to the subcomplex
 $\IDR (X\times M)$ defined in \ref{jun1302}.

In \cite{buta}, for a bundle $V$ with a  geometry $g$ we defined a characteristic form 
\begin{equation}\label{apr1901}
\omega((V,(h^V,\nabla^{II}))):= \prod_p (\ch_{2p}(\nabla^u)\oplus\ch_{2p}(\nabla), \tilde\ch_{2p-1}(\nabla^u,\nabla))
\end{equation}
where the last form denotes the transgression \cite[(66)]{buta}. This is compatible with our new construction in the sense of the lemma below.
We let $i\Vect^{geom}$ denote the symmetric monoidal stack of bundles with    geometries on $\Mf\times \Reg_{\Z}$ and geometry preserving isomorphisms.\footnote{Note that in \cite{buta} this symbol has a different meaning.}
Then the formula \eqref{apr1901} gives a map
$\omega\colon \pi_{0}(i\Vect^{ geom})\to Z^{0}(\DR)$. 
The construction of the canonical extended geometry in Definition \ref{jun2202} induces a map
$$
\can\colon\pi_{0}(i\Vect^{geom} )\to \pi_{0}( i\Vect^{exge}), 
$$
which is additive, but not multiplicative.

\begin{lem}\label{jun1305}
The diagram 
\begin{equation}\label{may2740}
\xymatrix{
\pi_{0}(i\Vect^{exge} )\ar[r]^-{{\widetilde \omega}}&Z^{0}(\IDR) \ar[d]^{{q}}\\
\pi_{0}(i\Vect^{geom} )\ar[u]^{{\can}} \ar[r]^-{\omega}&Z^{0}(\DR) 
}
\end{equation}
commutes.
\end{lem}
\begin{proof}
This follows from the definition of $q$ in \eqref{may2350}, the construction of the transgression
$\tilde\ch_{2p-1}(\nabla^u,\nabla)$, and the definition of $\omega$ in \eqref{apr1901}.
\end{proof}

\subsection{The multiplicative $K$-theory sheaf and the regulator}
\label{sec:regulator}

In this section, we define algebraic $K$-theory as a sheaf of commutative ring spectra on $\Mf \times \Reg_\Z$.
To do so, we use the multiplicative version of group completion studied in
\cite{2013arXiv1305.4550G} (see in particular their Proposition 8.2).
We denote by $\Sp^{\wedge}$ (resp.~$\Sp^{\geq 0,\wedge}$) the symmetric monoidal
$\infty$-category of spectra (resp.~connective spectra) with the
smash product. The category $\Sp$ is the stable $\infty$-category generated by the sphere spectrum whose homotopy category is the stable homotopy category. For the purpose of the present paper we do not have to fix a particular model for $\Sp$. We will use the identification of $\infty$-categories
$\CommGroup(\sSet[W^{-1}]^{\times})\simeq \Sp^{\ge 0,\wedge}$ which identifies a connective spectrum with its $\infty$-loop space. This equivalence refines to an equivalence of $\infty$-categories
\begin{equation}\label{kdaksjdklasd}
\Ring(\sSet[W^{-1}]^{\times})\simeq \CAlg(\Sp^{\ge 0,\wedge})\ .
\end{equation}

\begin{ddd}
We define the $K$-theory functor
\[
K\colon \SemiRing(\Cat[W^{-1}]^{{\times}})\to \CAlg(\Sp^{\wedge})
\]
as the composition
\begin{align*}
\SemiRing(\Cat[W^{-1}]^{\times})&\xrightarrow{\Nerve} \SemiRing(\sSet[W^{-1}]^{\times}) && \text{(Nerve)} \\
&\to \Ring(\sSet[W^{-1}]^{\times}) && \text{(Ring completion)}\\
&\xrightarrow{\simeq}  \CAlg(\Sp^{\ge 0,\wedge}) &&  {\text{(using \eqref{kdaksjdklasd})}}\\
&\to \CAlg(\Sp^{\wedge})&&\text{(forget connectivity)}  .
\end{align*}
\end{ddd}
We consider the sheaf 
$$i\Vect\in \Fun^{desc,I}((\Mf\times \Reg_{\Z})^{op},\SemiRing(\Cat[W^{-1}]^{\times}))$$   which associates to each object
$M\times X$ the Rig-category of bundles over $M\times X$ and isomorphisms.
\begin{ddd} \label{jun1320}
We define the sheaf of $K$-theory spectra by 
$$\bK:= L(K(i\Vect))\in 
\Fun^{desc,I}((\Mf\times \Reg_{\Z})^{op},\CAlg(\Sp^{\wedge}))$$
\end{ddd} 
\begin{rem}
For $X\in \Reg_\Z$, the homotopy groups of the spectrum $\bK(X) := \bK(*\times X)$ are the usual
$K$-groups
of $X$ as defined by Quillen. This follows from the known facts that, for affine $X$, Quillen's
$K$-theory coincides with $K$-theory defined by group completion and that, on $\Reg_\Z$, Quillen's
$K$-theory satisfies Zariski-descent (see \cite[Section 3.3]{buta} for more details).

In general, the spectrum $\bK(X)$ represents a generalized cohomology theory and, for a
manifold $M$, we have
\[
\pi_*\left(\bK(M\times X)\right) \cong \bK(X)^{-*}(M)
\]
(see~\cite[Section 4.5]{buta}).
\end{rem}

Note that the homotopy invariance of $i\Vect$ implies the homotopy invariance of {$K(i\Vect)$}.
In contrast, 
$i\Vect^{exge} $ is not homotopy invariant. But applying  the presheaf
homotopification $\bar\bs \simeq \cH^{pre}$ from \eqref{apr1902} we get:
\begin{lem}\label{may2710}
The natural `forget the geometry' map
$$
\bar \bs\  \Nerve(i\Vect^{exge})\to  \bar\bs \ \Nerve(i\Vect) \simeq \Nerve(i\Vect)
$$
is an equivalence in 
 $\Fun ((\Mf\times \Reg_{\Z})^{op},\SemiRing(\sSet[W^{-1}]^{\times}))$. 
\end{lem}
\begin{proof}
Since the colimit over $\Delta^{op }$ appearing in the definition \eqref{apr1902} of  $\bar \bs$ is sifted
it commutes with the forgetful functor $\SemiRing(\sSet[W^{-1}]^{\times})\to \sSet[W^{-1}]$. 
This follows from a two-fold application of \cite[Corollary 3.2.3.2.]{HA} to $\SemiRing(\sSet[W^{-1}]^{\times})\simeq \CAlg(\CAlg(\sSet[W^{-1}]^{\times})^{\otimes})$.
Since an equivalence in $\SemiRing(\sSet[W^{-1}]^{\times})$ is detected in $\sSet[W^{-1}]$ it suffices
to show that the induced map in 
 $\Fun ((\Mf\times \Reg_{\Z})^{op}, \sSet[W^{-1}])$ is an equivalence.

We claim that for $M\times X\in \Mf\times\Reg_\Z$  
the map of simplicial sets $$\bs\ \Nerve(i\Vect^{exge})(M\times X)_{\bullet,q} \to
\bs\ \Nerve(i\Vect)(M\times X)_{\bullet,q}$$ is a trivial Kan fibration. The result then follows by
applying the colimit as in \eqref{apr1902}.

A $p$-simplex $x\colon \Delta^p \to
\Nerve(i\Vect)(M\times X)_{\bullet,q}$ is given by a string of bundles and isomorphisms
\[
V_0 \xrightarrow{\cong} V_1 \xrightarrow{\cong} \dots \xrightarrow{\cong} V_q
\]
on $\Delta^p\times M \times X$. A lifting of $x|_{\partial\Delta^p}$ is determined by an extended
geometry on $V_0|_{\partial\Delta^p\times M \times X}$. Using the fact that extended geometries
exist and can be glued using partitions of unity, we see that such a lifting can always be extended
to a $p$-simplex of $\bs\ \Nerve(i\Vect^{exge})(M\times X)_{\bullet,q}$ lifting $x$.  This implies the claim. 
\end{proof}

We now turn to the construction of the multiplicative version of Beilinson's regulator.
We interpret a set as a discrete category. In this way we get a morphism
\[
\iota\colon \SemiRing(\Set ^{\times})\to  \SemiRing(\Cat[W^{-1}]^{\times})\ .
\]   
We have a commutative diagram (see \cite[Remark~2.13]{buta})
\[
\xymatrix{
\Ring(\Set^{{\times}})\ar[d]\ar[r]^-{\iota}\ar[d]^{S^{0}} &\SemiRing(\Cat[W^{-1}]^{\times})\ar[d]^{K}\\
\CAlg(\Ch[W^{-1}]^{\otimes})\ar[r]^-{H} & \CAlg(\Sp^{{\wedge}})
}
\]
where $S^{0}$ interprets a commutative ring as a commutative monoid in chain complexes concentrated  in degree zero,  $H$ is the Eilenberg-MacLane equivalence, and in the upper horizontal line we do not write the restriction of $\iota$ from semirings to rings explicitly.
We write $r({\widetilde \omega})$ for the composition
\begin{multline*}
K(i\Vect^{exge})\to K(\iota(\pi_0(i\Vect^{exge})))\xrightarrow{K(\iota({\widetilde \omega}))}
K(\iota(Z^{0}(\IDR))) \simeq \\ 
\simeq  H(S^{0}(Z^{0}(\IDR)))\to H(\IDR)
\end{multline*}
in $\Fun((\Mf\times \Reg_\Z)^{op}, \CAlg(\Sp^\wedge))$.

In analogy with \cite[Definition~4.36]{buta} we adopt the following definition:
\begin{ddd}\label{jul08161}
We define the multiplicative version of the  naive Beilinson regulator 
\[
\beil\colon \bK \to H(\IDR)
\]
as a morphism in $\Fun^{desc}((\Mf\times \Reg_\Z)^{op}, \CAlg(\Sp^\wedge))$ to be the
sheafification of the composition
\[
K(i\Vect) \xrightarrow{\simeq} \bar\bs\ K(i\Vect) \xleftarrow[\text{Lemma \ref{may2710}}]{\simeq} \bar
\bs\ K(i\Vect^{exge}) \xrightarrow{\bar \bs(r(\widetilde \omega))} \bar \bs \ H(\IDR)
\xleftarrow{\simeq} H(\IDR)
\]
in $\Fun((\Mf\times \Reg_\Z)^{op}, \CAlg(\Sp^\wedge))$.
\end{ddd} 
 {Here we use the fact that $H(\IDR)$ is a sheaf (see \eqref{descentIDR}).}

 {
\begin{rem}
Since  in the present paper we don't require geometries to be good in the sense of
\cite[Definition 4.17]{buta} the characteristic forms don't necessarily satisfy a
logarithmic growth condition at infinity. Therefore, we end up in analytic Deligne cohomology
instead of
absolute Hodge cohomology. 
The proof of Lemma \ref{may2710} does not work for good geometries. 
In \cite{buta} we found a way to avoid this problem using the \v{C}echification of the 
de Rham complexes. At the moment we do not see how to refine this to a multiplicative version.
\end{rem}
}

For $X \in \Reg_\Z$  Beilinson's regulator \cite{BeilinsonHodgeCohom} is a homomorphism from the
$K$-theory of $X$ to absolute Hodge cohomology (see Remark \ref{aug0601})
\[
K_*(X) \to \prod_p H^{2p-*}_{\cA\cH}(X,\R(p))\ .
\]
It is known to be multiplicative. We call its composition with the natural map 
\[
\prod_p H^{2p-*}_{\cA\cH}(X,\R(p)) \to H^{-*}\left(\IDR(*\times X)\right)
\]
the analytic version of Beilinson's regulator.
\begin{theorem}\label{may2760}
The naive Beilinson regulator 
\[
\beil\colon \bK \to H(\IDR)
\]
is a morphism of sheaves of ring spectra  which on the homotopy groups of its evaluation on $*\times
X$ induces the analytic version of Beilinson's regulator.
\end{theorem}
\begin{proof}
The first assertion is true by construction. 
It is also immediate from 
the constructions and Lemma \ref{jun1305} that the
map of sheaves of spectra underlying $\beil$ coincides with the one obtained in \cite[Definition
4.36]{buta} (after forgetting the logarithmic growth condition and using the equivalence $\DR\cong
\IDR$). 
For the latter, the coincidence with Beilinson's regulator was proven in \cite[Section~4.7]{buta}.
\end{proof}

\section{Multiplicative differential algebraic $K$-theory}\label{juni1010}

\subsection{Basic definitions}

The main goal of this section  is the definition of a multiplicative version of  differential algebraic $K$-theory for {objects in $\Mf \times \Reg_\Z$} and the verification of its basic properties.

For a complex $C\in \Ch$ and an integer $k$ we let $\sigma^{\ge k}C$ denote the naive truncation given by $\dots \to0\to C^{k}\to C^{k+1}\to \dots$. There is a natural inclusion morphism $\sigma^{\ge k}C\to C$.
\begin{ddd}
For every integer $k\in \Z$, we define
the sheaf of differential algebraic $K$-theory spectra 
$$
\hbK^{(k)} \in  \Fun^{desc}( {({\Mf\times \Reg_\Z})^{op}}, \Sp^{\wedge})
$$
by the pull-back
$$
\xymatrix@C+0.3cm{
\hbK^{(k)} \ar[r]^-{R}\ar[d]^{I} & H(\sigma^{\ge k} \IDR )\ar[d]\\
\bK\ar[r]^-{\beil} & H(\IDR).
}
$$
We define the differential algebraic $K$-theory for {objects in $\Mf \times \Reg_\Z$} as a presheaf of abelian groups
$$
{\hbK}\oben{k}:=\pi_{-k}(\hbK\oben{(k)})\in \Fun( {({\Mf\times \Reg_\Z})^{op}},\Ab)\ . $$
\end{ddd}
\begin{rem}
The integer $k\in \Z$ determines that the homotopy group
 $\pi_{-d}(\hbK^{k})$ for $d\in \Z$ captures interesting differential geometric information exactly if $d=k$.
\end{rem}

In the following, we refine $\bigvee_{k\in\Z} \hbK\oben{(k)}$ to a sheaf of commutative ring spectra (see \cite[Section 4.6]{skript} for details). 
Using the symmetric monoidal functors
\[
\Set \xrightarrow{\iota} \sSet[W^{-1}] \xrightarrow{\Sigma^{\infty}_{+}} \Sp
\]
the abelian group $\Z\in \CommMon(\Set)$ gives rise to the commutative ring spectrum
$\Sigma^{\infty}_{+}\iota(\Z) \in \CAlg(\Sp^{{\wedge}})$. For any commutative ring
spectrum $E$ we write $E[z,z^{-1}] := E \wedge \Sigma^{\infty}_{+}\iota(\Z)$.
We consider $\IDR[z,z^{-1}] := \IDR \otimes_{\Z} \Z[z,z^{-1}]$ as a sheaf of commutative differential graded algebras and define the subalgebra 
\[
\sigma^{\geq \bullet} \IDR := \bigoplus_{k\in\Z} z^{k} \sigma^{\geq k}\IDR \subseteq \IDR[z,z^{-1}].
\]
We have a natural equivalence $H(\IDR[z,z^{-1}]) \simeq H(\IDR)[z,z^{-1}]$.
\begin{ddd}
We define differential algebraic $K$-theory as a sheaf of commutative ring spectra 
\[
\hbK\oben{(\bullet)} \in \Fun^{desc}((\Mf\times\Reg_{\Z})^{op}, \CAlg(\Sp^{{\wedge}}))
\]
by the pull-back
\[
\xymatrix@C+0.9cm{
\hbK\oben{(\bullet)} \ar[r]^-{R}\ar[d]^{I} & H(\sigma^{\geq \bullet}\IDR) \ar[d] \\
\bK[z,z^{-1}] \ar[r]^-{\beil[z,z^{-1}]} & H(\IDR)[z,z^{-1}].
}
\]
\end{ddd}
If we forget the ring spectrum structure, then we get a natural equivalence $\hbK\oben{(\bullet)}
\simeq
\bigvee_{k\in\Z} \hbK\oben{(k)}$. In particular, we get a presheaf of graded commutative rings
\[
\bigoplus_{k\in\Z} \hbK\oben{k} \in \Fun((\Mf\times \Reg_\Z)^{op}, \GrRings).
\]
The maps $R$ and $I$ induce ring homomorphisms
\[
R\colon \bigoplus_{k\in\Z} \hbK\oben{k} \to \bigoplus_{k\in\Z} Z^k(\IDR), \quad 
I\colon \bigoplus_{k\in\Z} \hbK\oben{k} \to \bigoplus_{k\in\Z} \bK^k .
\]
The map $R$ is called the curvature.
For any $k\in \Z$ we have exact sequences
\[
\bK^{k-1} \xrightarrow{\beil} H^{k-1}(\IDR ) \xrightarrow{a} {\hbK}\oben{k} \xrightarrow{(I,R)}
\bK^{k}\times_{H^{k}(\IDR )}  Z^{k}(\IDR )\to 0
\]
and  
\begin{equation}\label{apr2503}
{\bK^{k-1} \xrightarrow{\beil}} \IDR^{k-1}/\im(d) \xrightarrow{a} {\hbK}\oben{k} \xrightarrow{I}
\bK^{k}\to 0
\end{equation}
(see \cite[Proposition 5.4]{buta}).
Moreover, we have the relation $R\circ a = d$.

\subsection{Cycle maps}

We have the forgetful map
$$
\pi_0 (i\vect^{exge})\to  \pi_0(i\vect)
$$ 
between the presheaves of  semirings of isomorphism classes of bundles with and without extended
geometries.

\begin{prop} There are canonical cycle maps {$\cycl$ and $\hcycl$} fitting into the following diagram of
presheaves of semirings on ${\Mf\times \Reg_\Z}$:
$$
\xymatrix@C+0.5cm{
\pi_0(  i\vect^{exge}) \ar@/^2pc/[rr]^{\widetilde\omega} \ar[r]^-{\hcycl}\ar[d]&{\hbK}^{0}\ar[d]^{I} {\ar[r]^-R}& {Z^0(\IDR)}\\
\pi_0(i\vect)\ar[r]^-{\cycl}&\bK^{0}
}.
$$
\end{prop}
\begin{proof}
The construction is identical to that of \cite[Definitions~5.8, 5.9]{buta}. 
\end{proof}

\subsection{$S^{1}$-integration}

We consider $M\times X \in \Mf\times \Reg_{\Z}$. 
Let $\bE\in \Fun^{desc,I}((\Mf\times\Reg_{\Z})^{op}, \Sp)$ be a homotopy invariant sheaf of spectra.
Then we have natural isomorphisms
\[
\bE^{*}(S^{1}\times M \times X) \cong \bE^{*}(M\times X) \oplus \bE^{*-1}(M\times X).
\]
The induced map $\bE^{*}(S^{1}\times M \times X) \to \bE^{*-1}(M\times X)$ is called the desuspension map. 
This applies in particular to the K-theory sheaf $\bK$ and the analytic Deligne cohomology $H(\IDR)$.

On the other hand, on the level of differential forms we have the usual fibre integration along $S^{1}$, a map of complexes
\[
\int_{S^{1}}\colon \IDR(S^{1}\times M \times X) \to \IDR(M\times X)[-1].
\]
It induces integration maps $\int_{S^{1}}\colon \sigma^{\geq k}\IDR(S^{1}\times M \times X) \to
\sigma^{\geq k-1} \IDR(M\times X)[-1]$ for any $k\in \Z$.

\begin{prop}\label{djhqwkdhqwkdqwd90}
There exists a natural map
\[
\int_{S^{1}}\colon \hbK\oben{*}(S^{1}\times M \times X) \to \hbK\oben{*-1}(M\times X)
\]
of $\bigoplus_{k\in\Z}\hbK\oben{k}(M\times X)$-modules which is compatible with the desuspension on $\bK^{*}$ via the map $I$ and with the integration $\int_{S^{1}}$ on $Z^{*}(\IDR)$ via the curvature $R$.
\end{prop}
\begin{proof}
For any presentable $\infty$-category $\bC$ 
we define the endofunctor $S^{1}$ of
$\Fun^{desc}(\Mf^{op},\bC)$ such that
$(S^{1}F)(M\times X):=F(S^{1}\times M\times X)$.
If $\bC$ is symmetric monoidal and $F\in \Fun^{desc}(\Mf^{op},\CAlg(\bC))$,
then the projection {$\pr\colon {S^{1}}\to *$} turns $S^{1}F$ into an object of
$\Mod(F)$.

We extend the endofunctor $S^{1}$ to
$\Fun^{desc}((\Mf\times \Reg_{\Z})^{op},\bC)$ using the identification
$$
\Fun^{desc}((\Mf\times \Reg_{\Z})^{op},\bC) \simeq
\Fun^{desc}(\Mf^{op},\Fun^{desc}(\Reg_{\Z}^{op},\bC))\ .
$$

The evaluation at the manifold $M=*$ provides an equivalence of  $\infty$-categories
\begin{equation}\label{sep1203n}
\ev_{*}:\Fun^{desc,I}(\Mf^{op},\bC)\xrightarrow{\simeq}  \bC\ ,
\end{equation}
 and
we have an equivalence of functors $\Fun^{desc,I}(\Mf^{op},\bC) \to  \bC$
\begin{equation}\label{sep1501}
\ev_{*}\circ S^{1}(-)\simeq (\ev_{*}(-))^{S^{1}}\ ,
\end{equation} 
where
$(-)^{S^{1}}$ is the cotensor structure. Let $\pr\colon S^{1}\to *$ and $i\colon *\to S^{1}$ be the
projection to a point and the inclusion of a base point. These maps induce a retraction
$$\id(-)\stackrel{\pr^{*}}{\to} (-)^{S^{1}}\stackrel{i^{*}}{\to} \id(-)\ .$$
If $\bC$ is stable, then we can naturally split off $\id(-)$ as a summand of $(-)^{S^{1}}$ and identify the complement with $\Omega (-)$. The desuspension map is by definition   the projection
\begin{equation}\label{sep1502}
\des:(-)^{S^{1}}\to \Omega(-)\ .
\end{equation}
Under the equivalence \eqref{sep1203n} in the case $\bC=\Fun^{desc}(\Reg_{\Z},\Sp)$ it induces the desuspension map in cohomology mentioned above.

The integration of forms gives  morphisms of sheaves with values in $\Ch$
$$\int_{S^{1}}\colon S^{1} \IDR\to \IDR[-1]\ , \quad \int_{S^{1}}\colon S^{1}\sigma^{\ge k}\IDR\to
\sigma^{\ge k-1} \IDR[-1]
$$ 
which, when assembled for the various $k\in \Z$, after application of the Eilenberg-MacLane functor $H$ yield the commutative diagram
 \begin{equation}\label{sep1201n}\begin{split}
\xymatrix@C+0.5cm{
H(S^{1}\sigma^{\ge \bullet}\IDR)\ar[r]^{H(\int_{S^{1}})}\ar[d]&\Omega H(\sigma^{\ge \bullet-1}\IDR) \ar[d]\\
H(S^{1}\IDR[z,z^{-1}])\ar[r]^{H(\int_{S^{1}})}\ar[r]&\Omega H(\IDR[z,z^{-1}])
}
\end{split}
\end{equation}
in $\Mod(\hbK\oben{(\bullet)} )$, where $\hbK\oben{(\bullet)}$  acts via the curvature map.
From the naturality of the desuspension we get the commutative diagram
\begin{equation}\label{sep1202n}\begin{split}
\xymatrix{
S^{1}\bK[z,z^{-1}]\ar[r]^{\des}\ar[d]^{\beil}&\Omega \bK[z,z^{-1}]\ar[d]^{\Omega \beil}\\
S^{1}H(\IDR[z,z^{-1}])\ar[r]^{\des}&\Omega H(\IDR[z,z^{-1}])
}
\end{split}
\end{equation}
in $\Mod(\hbK\oben{(\bullet)})$, where here $\hbK\oben{(\bullet)}$ acts via $I$.
\begin{lem}\label{sep1503}
We have a natural equivalence of morphisms
\[
\des \simeq H(\int_{S^{1}}):S^{1} H(\IDR[z,z^{-1}])\to \Omega H(\IDR[z,z^{-1}])
\]
in $\Mod(\hbK\oben{(\bullet)})$.
 \end{lem} 
Before proving this Lemma we finish the argument for Proposition \ref{djhqwkdhqwkdqwd90}.
Together with \eqref{sep1202n}, lemma \ref{sep1503} provides the lower square of the
following diagram in $\Mod(\hbK\oben{(\bullet)})$
 $$
\xymatrix@C+0.5cm{
H(S^{1}\sigma^{\ge \bullet}\IDR)\ar[r]^{H(\int_{S^{1}})}\ar[d]&\Omega H(\sigma^{\ge \bullet -1}\IDR) \ar[d]\\
H(S^{1}\IDR[z,z^{-1}])\ar[r]^{H(\int_{S^{1}})}\ar[r]&\Omega H(\IDR[z,z^{-1}])\\
S^{1}\bK[z,z^{-1}]\ar[u]^{\beil}\ar[r]^{\des}&\Omega \bK[z,z^{-1}]\ar[u]^{\Omega \beil}
}\ .
$$
The upper square is \eqref{sep1201n}.
In view of the definition of $\hbK\oben{(\bullet)}$ as a pull-back this diagram induces a map
$$\int_{S^{1}}:S^{1}\hbK\oben{(\bullet)} \to \Omega \:\hbK\oben{(\bullet)}$$
in $\Mod(\hbK\oben{(\bullet)})$.
It induces the asserted integration map in cohomology. 
\end{proof} 

\begin{proof}[Proof of Lemma \ref{sep1503}]
We have a natural equivalence in $\Mod(\IDR[z,z^{-1}])$
\[
\IDR[z,z^{-1}][-1]\oplus \IDR[z,z^{-1}] \xrightarrow{\sim} S^1\IDR[z,z^{-1}]
\]
given on $M\times X$ by $\omega\oplus\eta \mapsto dt\wedge \pr^*\omega + \pr^*\eta$, where $t$ is
the coordinate on $S^1$ and $\pr\colon S^1\times M \times X\to M\times X$ is the projection.
An explicit inverse is given by $(\int_{S^1}, i^*)$ where $i\colon M \times X \to S^1\times M
\times X$ is induced by the inclusion of a point in $S^1$.
In view of the definition of the desuspension in \eqref{sep1502} and the equivalence
\eqref{sep1501} we can identify the desuspension for $\IDR[z,z^{-1}]$ naturally with the map
$\int_{S^1}\colon S^1\IDR[z,z^{-1}] \to \IDR[z,z^{-1}][-1]$. 
Now the result follows by applying $H$.
\end{proof}

\section{A secondary Steinberg relation}\label{juni1002}

\subsection{Units}

Let $R$ be a ring such that $X=\Spec(R)\in \Reg_\Z$. We have a natural homomorphism 
\begin{equation}
\label{eq:c}
c\colon R^{\times} \to \bK^{-1}(X)
\end{equation}
where we write $\bK^{-1}(X)$ instead of $\bK^{-1}(\ast\times X)$. Concretely, $c$ is given as
follows: For $\lambda\in R^{\times}$ we let $\cV(\lambda)$ be the bundle on $S^1\times X$ which
restricts to the trivial bundle $\cO_X$ at any point $t\in S^1$ and has holonomy $\lambda$ along
$S^1$. Then 
\begin{equation}\label{aug0202}
\cycl(\cV(\lambda)) = c(\lambda) \oplus 1 \in \bK^0(S^1\times X) \cong \bK^{-1}(X) \oplus \bK^0(X) \ .
\end{equation}
Since the kernel of the map $I\colon \hbK^{-1}(X) \to \bK^{-1}(X)$ is a divisible abelian group,
there exists a lift $\hat c\colon R^{\times} \to \hbK^{-1}(X)$ of $c$. In the following, we will fix
a specific choice of this lift.

We first construct a
geometry $(h^{(\lambda)},\nabla^{(\lambda)})$ on $\cV(\lambda)$. 
Abusing notation, we also denote the complex line bundle on $S^1\times X(\C)$ associated with
$\cV(\lambda)$ by the same symbol and view $\lambda$ as a nowhere vanishing function on $X(\C)$. Let
$t$ be a parameter on $S^1$ and $\log(\lambda)$ a local choice of a logarithm of $\lambda$ on
$X(\C)$. Then $\phi = \lambda^t$ is a local section of $\cV(\lambda)$ which depends on the
choice of logarithm. The metric and the connection are determined by their value on the local
sections $\phi$. We set
\begin{align} \label{aug0201}
& h^{(\lambda)}(\phi) = 1, \\
& \nabla^{(\lambda)}(\phi) = \log(\lambda)\phi \:dt\ . \notag
\end{align}
These are well defined. Moreover, $\nabla^{(\lambda)}$ has holonomy $\lambda$ along $S^1$ and
$[\nabla, \bar\partial]=0$. 
We equip $\cV(\lambda)$ with the canonical extended geometry, denoted by $g(\lambda)$.
\begin{ddd}
We define $\hat c\colon R^{\times} \to \hbK\oben{-1}(X)$ to be the composition 
\[
\hat c\colon R^{\times} \xrightarrow{\lambda\mapsto \hcycl(\cV(\lambda), g(\lambda))} \hbK\oben{0}(S^1\times X)
\xrightarrow{\int_{S^1}} \hbK\oben{-1}(X)\ .
\]
\end{ddd}
\begin{lem}\label{aug0501}
The curvature $R(\hat c(\lambda)) \in Z^{-1}(\IDR(X))$ is given
by
\[
R(\hat c(\lambda)) = R(\hat c(\lambda))(1) = id\arg(\lambda) + d\log(|\lambda|^u) \in
Z^{-1}(\IDR(1)(X)) \subset A^1(I\times X(\C)),
\] 
where $u$ is the coordinate on the interval $I$. The induced map
\[
\hat c\colon R^{\times} \to \hbK\oben{-1}(X)/a(H^{-2}(\IDR(X)))
\]
is a homomorphism.
\end{lem}
\begin{proof}
For the adjoint connection of $\nabla^{(\lambda)}$ we get
\[
\nabla^{(\lambda),*}\phi = -\log(\bar\lambda) \phi \:dt\ .
\]
Hence the connection of the canonical extended geometry is given by 
\[
\widetilde\nabla^{(\lambda)}\phi = \left(\frac{1-u}{2}(\log(\lambda) - \log(\bar\lambda)) +
u\log(\lambda)\right)\phi \:dt\ .
\]
Together with \eqref{aug0201} this implies that for two units $\lambda, \mu \in R^{\times}$ we have
$$(\cV(\lambda\mu), g(\lambda\mu)) \cong (\cV(\lambda), g(\lambda)) \otimes (\cV(\mu),g(\mu))\ .$$ By
the multiplicativity of the geometric cycle map we get
\[
\hcycl(\cV(\lambda\mu), g(\lambda\mu)) =\hcycl(\cV(\lambda), g(\lambda))
\cup \hcycl(\cV(\mu),g(\mu))\ .
\]
For the curvature we get
\[
R^{\widetilde\nabla^{(\lambda)}} = -i dt\wedge d\arg(\lambda) - dt \wedge d\log(|\lambda|^u).
\]
Hence 
\begin{multline}\label{aug0203}
R(\hcycl(\cV(\lambda),g(\lambda))) = 1 \oplus (i dt\wedge d\arg(\lambda) + dt \wedge
d\log(|\lambda|^u)) \\
\in Z^0(\IDR(0)(S^1\times X)) \oplus Z^0(\IDR(1)(S^1\times X)). 
\end{multline}
Integration over $S^1$ kills the first summand and gives the statement about the curvature.

From the formula for the curvature and the fact that $ {c=I\circ \hat c}$  {(see
\eqref{eq:c})} is a homomorphism, we get 
\[
R(\hat c(\lambda\mu)) = R(\hat c(\lambda)) + R(\hat c(\mu)), \quad 
I(\hat c(\lambda\mu)) = I(\hat c(\lambda)) + I(\hat c(\mu))\  ,
\]
hence $\hat c(\lambda\mu) - \hat c(\lambda) - \hat c(\mu) \in a(H^{-2}(\IDR(X)))$. 
\end{proof}

\subsection{The Steinberg relation and the Bloch-Wigner function}

In this subsection we explain how differential algebraic $K$-theory can be used to give a simple
proof of a result of Bloch {(see~\cite{Bloch-HigherReg})} concerning the existence of
classes in $K_{3}$ of a number ring whose regulator can be described in terms of the Bloch-Wigner
dilogarithm function. The key ingredient is a secondary version of the Steinberg relation.

We begin by collecting some notation necessary to state the result. Recall the definition of the polylogarithm functions $$\Li_k(z):=\sum_{n\ge 1} \frac{z^{n}}{n^{k}}$$
for $k\ge 1$ and $|z|<1$. They extend meromorphically to a covering of $\C\setminus\{1\}$.
\begin{ddd} The Bloch-Wigner function is the real {valued} function on $\C$ given by 
$$D^{BW}(\lambda):= \log |\lambda| \arg(1-\lambda) +\Im\Li_2(\lambda) $$
(see \cite[{Ch.~I, \S 3}]{MR2290758}).
\end{ddd}

Let $R$ be a ring.
\begin{ddd}
We write $R^{\circ}:= \{\lambda\in R^{\times}\,|\,1-\lambda\in R^{\times}\}$.
The third \emph{Bloch group} $\cB_{3}(R)$ is defined as the kernel
\[
\cB_{3}(R) := \ker\left(\Z[R^{\circ}] \xrightarrow{\lambda\mapsto \lambda\wedge (1-\lambda)} R^{\times}\wedge R^{\times}\right).
\]
\end{ddd}
Now let $R$ be the ring of integers in a number field and $X:=\Spec(R)$. 
The target of the regulator $\beil$ on $\bK^{-3}(X)$ is $H^{-3}(\IDR(X))$. Since $X(\C)$ is zero dimensional we have
\begin{equation}\label{aug0502}
H^{-3}(\IDR(X)) \cong H^{-3}(\IDR(2)(X)) \underset{q}{\cong} H^{-3}(\DR(2)(X)) \cong \left[2\pi i\R^{X(\C)}\right]^{\Gal(\C/\R)}.
\end{equation}
\begin{theorem}[Bloch]\label{jdkdkdqwdqwdqwdwqd}
For any $x = \sum_{\lambda\in R^{\circ}} n_{\lambda}[\lambda]\in \cB_{3}(R)$, there exists an element $bl(x) \in \bK^{-3}(X)$ such
that, under the identification \eqref{aug0502}, 
\[
\beil(bl(x)) = -\sum_{\lambda} n_{\lambda} \left( iD^{BW}(\sigma(\lambda))\right)_{\sigma\in X(\C)} \ .
\]
\end{theorem}
\begin{ex}
Assume that $n\in \nat$, $n\ge 2$ and $\lambda\in R^{*}$ satisfies
$$
\lambda^{n+1}-\lambda+1=0\ .
$$
Then {$\frac{1}{1-\lambda} \in R^{\circ}$ and we}
consider the element ${x:=n[\lambda]+ [\frac{1}{1-\lambda}] }\in \Z[R^{\circ}]$. We claim that $x\in \cB_{3}(R)$. Indeed,
\begin{eqnarray*}
n (\lambda\wedge (1-\lambda))+ \frac{1}{1-\lambda} \wedge(1-\frac{1}{1-\lambda}) &=&
n(\lambda\wedge (1-\lambda))+ \frac{1}{1-\lambda} \wedge
\frac{\lambda}{\lambda-1} \\
&=& \lambda^n\wedge (1-\lambda) + (1-\lambda) \wedge \frac{\lambda-1}{\lambda}\\
&=& \frac{\lambda -1}{\lambda} \wedge(1-\lambda) + (1-\lambda) \wedge \frac{\lambda-1}{\lambda} \\
&=& 0.
\end{eqnarray*}
We get an element $bl(x)\in \bK_{3}(R)$ such that
$$\beil(2)(bl(x))=(n+1) \left(-iD^{BW}(\sigma(\lambda))\right)_{\sigma\in\Spec(R)(\C)}\ ,$$
where we use that {$D^{BW}(\frac{1}{1-\lambda})=D^{BW}(\lambda)$}.  If $\sigma(\lambda)$ is not real, then $D^{BW}(\sigma(\lambda))$ is not zero.
\end{ex}
\begin{proof}[Proof of Theorem \ref{jdkdkdqwdqwdqwdwqd}]
Since $X(\C)$ is zero dimensional  we have
$H^{-2}(\IDR(X)) = 0$. Hence, by Lemma \ref{aug0501}, the map $\hat c\colon R^{\times} \to \hbK\oben{-1}(X)$ is a homomorphism. Since $\bigoplus_{k\in \Z} \hbK\oben{k}(X)$ is graded commutative, we get an induced map $R^{\times} \wedge R^{\times} \to \hbK\oben{-2}(X)$, $\lambda\wedge\mu \mapsto \hat c(\lambda) \cup \hat c(\mu)$.

If $\lambda\in R^{\circ}$, then the Steinberg relation implies that
\[
I(\hat c(\lambda) \cup \hat c(1- \lambda)) = c(\lambda)\cup c(1- \lambda) = 0 \text{ in } \bK^{-2}(X).
\] 
Consider the following commutative diagram with exact rows:
\begin{equation}\label{aug0505}
\begin{split}
\xymatrix@C+0.2cm{
0 \ar[r] & \cB_{3}(R) \ar[r] \ar@{..>}[d]_{bl} & \Z[R^{\circ}] \ar[r]^-{\lambda\mapsto \lambda\wedge(1-\lambda)} \ar@{..>}[d]_{\cD}& R^{\times} \wedge R^{\times} \ar[d]\\
0 \ar[r] & \bK^{-3}(X)/\ker(\beil) \ar[r]^{\beil} & \IDR^{-3}(X)/\im(d) \ar[r]^-{a} & \hbK\oben{-2}(X) \ar[r]^{I} & \bK^{-2}(X) 
}
\end{split}
\end{equation}
The dotted arrow $\cD$ exists by the Steinberg relation and since $\Z[R^{\circ}]$ is a free abelian group. The dotted arrow $bl$ is the induced map on kernels.

We will now pin down a specific choice for $\cD$ which will then imply the Theorem.
\newcommand{\X}{\mathbb{X}}
To do this, we consider the universal situation. Let 
\[
\X:=\P^{1}_\Z\setminus \{0,1,\infty\} \cong \Spec(\Z{[\lambda,\lambda^{-1},(1-\lambda)^{-1}]}) .
\]
We consider $\hat c(\lambda) \cup \hat c(1-\lambda) \in \hbK^{-2}(\X)$. Again, by the Steinberg relation there exists $\cD(\lambda)\in \IDR^{-3}(\X)/\im(d)$ such that $a(\cD(\lambda)) = \hat c(\lambda) \cup \hat c(1-\lambda)$.
Since $R\circ a = d$, we must have
\begin{equation}\label{aug0503}
d(\cD(\lambda)) = R(\hat c(\lambda)) \cup R(\hat c(1-\lambda)) \in \IDR^{-2}(\X).
\end{equation}
Because we want to specialize to number rings later on, we are only interested in the component $\cD(\lambda)(2) \in \IDR(2)^{-3}(\X)$ (see~\eqref{aug0502})
This is determined by \eqref{aug0503} up to elements in $H^{-3}(\IDR(2)(\X))$.
Since $\cF^{2}A(I\times \X(\C)) = 0$ we have  quasi-isomorphisms
\begin{align}\label{aug0504}
\begin{split}
\IDR(2)(\X) & \underset{q}{\simeq}  \DR(2)(\X) \\
&\cong \left(\Cone\left((2\pi i)^{2}A_{\R}(\X(\C)) \to A(\X(\C))\right)[3]\right)^{\Gal(\C/\R)}\\
&\cong \left((2\pi i)A_{\R}(\X(\C))[3]\right)^{\Gal(\C/\R)}\ ,
\end{split}
\end{align}
where the last isomorphism is induced by taking $i$ times the imaginary part on the second component of the cone.
In particular, $H^{-3}(\IDR(2)(\X)) = H^{0}(\X(\C), (2\pi i)\R)^{\Gal(\C/\R)} = 0$. 

We now compute the right-hand side of \eqref{aug0503}. From Lemma \ref{aug0501} we get
\begin{multline*}
i\Im\left(R(\hat c(\lambda)) \cup R(\hat c(1-\lambda))\right) \\ = id\arg(\lambda) \wedge d\log(|1-\lambda|^u) +i d\log(|\lambda|^u))\wedge d\arg(1-\lambda).
\end{multline*}
Hence, under the quasi-isomorphisms \eqref{aug0504}, $R(\hat c(\lambda)) \cup  R(\hat c(1-\lambda))$ is mapped to
\[
i\log(|1-\lambda|) d\arg(\lambda) -i \log(|\lambda|)d\arg(1-\lambda)  \in \left((2\pi i)A^{1}_{\R}(\X(\C))\right)^{\Gal(\C/\R)}\ .
\]
On the other hand, using $\frac{d}{dz}\Li_{2}(z) = \frac{1}{z} \Li_{1}(z) = -\frac{1}{z}\log(1-z)$ we get
\begin{align*}
dD^{BW}(\lambda) &=  \arg(1-\lambda)d\log(|\lambda|) + \log(|\lambda|)d\arg(1-\lambda) - \Im \log(1-\lambda) d\log(\lambda) \\
& = \log(|\lambda|)d\arg(1-\lambda) - \log(|1-\lambda|)d\arg(\lambda)\ .
\end{align*}
It follows that, under the quasi-isomorphisms \eqref{aug0504},
\[
\cD(\lambda)(2) = -iD^{BW}(\lambda).
\]

We now return to the number ring $R$. Note that in diagram \eqref{aug0505} we may identify 
\begin{align*}
\IDR^{-3}(X)/\im(d) &= H^{-3}(\IDR(2)(X)) \\
&\cong \left((2\pi i)A_{\R}^{0}(X(\C))\right)^{\Gal(\C/\R)} \\
& = \left[2\pi i\R^{X(\C)}\right]^{\Gal(\C/\R)}.
\end{align*}
Any $\lambda\in R^{\circ}$ corresponds to a unique morphism $\lambda\colon X \to \X$, which on $\C$-valued points is given by $X(\C) \to \X(\C)=\C^{\times}\setminus \{1\}$, $\sigma \mapsto \sigma(\lambda)$. We construct $\cD(\lambda) \in \left[2\pi i\R^{X(\C)}\right]^{\Gal(\C/\R)}$ by pulling back along $\lambda$ from the universal case on $\X$. Explicitly, we get
\[
\cD(\lambda) = \left(-iD^{BW}(\sigma(\lambda))\right)_{\sigma\in X(\C)}.
\] 
This implies the formula for $bl$ stated in the Theorem. 
\end{proof}

\section{A height invariant for number rings}\label{juni1003}

{ 
Let $R$ be the ring of integers in a number field. We recall the following definition from Arakelov
geometry:
\begin{ddd}
A metrized line bundle $(\cL,h^{\cL})$ on $\Spec(R)$ is an invertible sheaf $\cL$ on $\Spec(R)$ with
a $\Gal(\C/\R)$-invariant metric $h^{\cL}$ on its complexification. We let $\hPic(\Spec(R))$
denote the multiplicative group of isomorphism classes of metrized line bundles
under the tensor product and call it the arithmetic Picard group of $R$.
\end{ddd}
We may identify $\cL$ with its $R$-module of global sections. A metric $h^{\cL}$ is then given by a
collection of metrics $h^{\cL}_{\sigma}$ on $\cL\otimes_{R,\sigma} \C$ for all $\sigma\in
\Spec(R)(\C)$ which is invariant under the $\Gal(\C/\R)$-action.
}

{ 
An important invariant is the arithmetic degree 
\[
\hdeg\colon \hPic(\Spec(R))\to \R
\]
defined as follows
(see \cite[IV,§ 3]{MR969124}): Let $(\cL, h^{\cL})$ be a metrized line bundle.  Then
\begin{equation}\label{aug0101}
\hdeg((\cL,h^{\cL})) := \frac{1}{[K:\Q]}\left( \log(\# (\cL/R\cdot s)) -
\frac{1}{2}\sum_{\sigma\in\Spec(R)(\C)} \log(h_{\sigma}(s))\right)
\end{equation}
where $s\in  \cL\setminus \{0\}$ is any non-zero section.}

{ 
The main aim of this section is to explain how the arithmetic Picard group and the arithmetic
degree can be naturally understood in the framework of differential algebraic $K$-theory (see
Theorem \ref{thmArith}).}

\subsection{Scaling the metric}\label{may0801}

Let $M$ be a smooth manifold and $X\in \Reg_\Z$.
We consider a geometric bundle $( V,g)$, $g:=(h^{ V},\nabla^{II})$, on $M\times X$ and let $f\in
C^{\infty}(M\times  X(\C) )$ be a $\Gal(\C/\R)$-invariant positive smooth function.
Then we can consider the rescaled metric
$f h^{ V}$ and geometry $g_f:=( f h^{ V},\nabla^{II})$.  In the following we work with the canonical extensions $\can(g)$ ({see~}Definition \ref{jun2202}) of the geometries.
We are interested in the difference
$$\hcycl(V,\can(g_f) )-\hcycl(V,\can(g))\in \widehat{\bK}(X)^{0}(M)\ .$$
Note that this difference is equal to
$ a(\alpha)$ for some $\alpha\in \IDR^{-1}(M\times X){/\im(d)}$, where
  $\alpha$ is well-defined up to the image of $\beil$.
We {want to} calculate $\alpha$.  To this end we use the homotopy formula {\cite[Lemma 5.11]{buta}}. We 
consider the bundle
$\widehat{{V}}:=\pr^{*} {V}$, where  $\pr\colon [0,1]\times M\times X\to M\times X$
is the projection. It is   equipped 
with the  geometry $\hat g:=(\hat h,\pr^{*}\nabla^{II})$, 
$\hat h:=(1-x+xf)h$, where $x\in [0,1]$ is the coordinate.
By the homotopy formula we can take
$$ 
\alpha= {\int_{[0,1]\times [0,1]\times M\times X/[0,1]\times M\times X} R(\hcycl(\widehat V, \can(\hat g))) = }
 \int_{[0,1]\times [0,1]\times M\times X/[0,1]\times M\times X} {\widetilde \omega}(\can(\hat g) ) \ .
$$
For us, the most important case is {the following (see~\cite[Lemma 5.13]{buta}):
 \begin{lem}\label{jun2211}
{If $\dim(M)=0$ and $\dim(X(\C))=0$,}
 we can take 
$$\alpha=\alpha(1)= {-}\frac{{\rk(V)}}{2}\log(f) du\ .$$
 \end{lem}
 \begin{proof}
 We have
  ${\widetilde \omega}(\hat g)(p)=0$ for all $p$ except $p=0,1$.  
 In fact we have 
$$
{\widetilde \omega}(\hat g)(0) \equiv {\rk(V)},
$$
{hence $\alpha(0) = 0$.}
 In order to calculate ${\widetilde \omega}(\hat g)(1)$, we first observe that 
$$
{\widetilde \nabla} =d+ {\frac{1-u}{2}}  \frac{(f -1)dx}{1+(f -1)x}\ .
$$
 We get
$$
{\widetilde \omega}(\hat h)(1) = \frac{{\rk(V)}}{2}\frac{(f -1)}{1+(f -1)x}du\wedge dx
$$
and therefore
\[
\alpha= \alpha(1)= {-}\frac{{\rk(V)}}{2}\log(f) du\ . \qedhere
\]
\end{proof}

\subsection{The absolute height for number rings}

We consider a ring of integers $R$ in a number field {$K$}. Note that
$\Spec(R)$ is regular, separated and  of finite type over $\Spec(\Z)$.
We define the multiplicative subgroup
$$
\bK^{0}(\Spec(R))_{(1)}:=\{x\in \bK^{0}(\Spec(R))\:|\: 1-x \:\:\mbox{is nilpotent}\}
$$
of the {group of} units in the ring  ${\bK^{0}(\Spec(R))}$. It is known that
$$\bK^{0}(\Spec(R))\cong \Z\oplus \Cl(R)\ ,$$
where $\Cl(R)$ denotes the finite class group.
Therefore
$$\bK^{0}(\Spec(R))_{(1)}\cong \{1+x\:|\: x\in \Cl(R)\}\cong \Cl(R)$$
is finite.
We furthermore define
$$
\hbK\oben{0}(\Spec(R))_{(1)}:=I^{-1}(\bK^{0}(\Spec(R))_{(1)})\subseteq \hbK\oben{0}(\Spec(R))\ . 
$$
If $x\in \widehat{\bK}^{0}(\Spec(R))_{(1)}$, then necessarily $R(x)=R(\beins)$. Hence we have an exact sequence
\begin{equation}\label{jul0501}
0\to  H^{-1}(\IDR (\Spec(R)))/\im(\beil) \xrightarrow{1+a} \hbK\oben{0}(\Spec(R))_{(1)}\to
\bK^{0}(\Spec(R))_{(1)} \to 0\ .
\end{equation}

 We now define an absolute height function 
$$
h\colon \hbK\oben{0}(\Spec(R))_{(1)} \to  \R
$$
for number rings $R$. 
{We will relate $h$ with the arithmetic degree of metrized line bundles in the next
subsection.}
 
Note that 
$$
H^{-1}(\IDR(\Spec(R)))\cong H^{-1}(\IDR(1)(\Spec(R)))   \cong
[\R^{\Spec(R)(\C)}]^{\Gal(\C/\R)}\ .
$$
Explicitly, a class $[\alpha]\in  H^{-1}(\IDR(1)(\Spec(R)))$ which is represented by 
$$\alpha\in \IDR(1)^{-1}({\Spec(R)}) \subseteq {A}^{1}([0,1]\times \Spec(R)(\C))$$
corresponds to the function
\begin{equation}\label{jun2210}
\Spec(R)(\C)\ni \sigma  \mapsto \Re(  \int_{[0,1]} {\sigma^*}\alpha)\in \R\ .
\end{equation}
We define a linear map
$$
s\colon  [\R^{\Spec(R)(\C)}]^{\Gal(\C/\R)}\to \R\ , \quad
s(f):=\frac{1}{{[K:\Q]}}\sum_{\sigma\in \Spec(R)(\C)}  f(\sigma)\ .$$
Then $s\circ \beil(1)=0$. 
In this way we get a homomorphism
\begin{equation}\label{jul0601}
{h}\colon H^{-1}(\IDR (\Spec(R)))/\im(\beil)\to \R\ , \quad  {h}([f]):=s(f)\ .
\end{equation}
In view of \eqref{jul0501} and since $\bK^{0}(\Spec(R))_{(1)}$ is finite, the homomorphism \eqref{jul0601} has a unique extension to  
$\widehat{\bK}^{0}(\Spec(R))_{(1)}$. Explicitly, if $x\in \widehat{\bK}^{0}(\Spec(R))_{(1)}$, then
there exists $N\in \nat $ such that $x^{N}=1+a(f)$
for some $f\in
H^{-1}(\IDR(\Spec(R)))$ {and ${h}(x)$} is given by 
$${h}(x)=\frac{1}{N} {h}(1+a(f))\ .$$


%

\subsection{{The degree of metrized line bundles}}

We  let $R$ be the ring of integers in a number field $K$. We
consider the trivial bundle ${\cV} := \cO_{\Spec(R)}$ with the canonical geometry $g_0 $. Then
$$
\hcycl(\cV,\can(g_0))=\beins\ .
$$
Let $f\colon\Spec(R)(\C)\to \R^{+}$ be $\Gal(\C/\R)$-invariant and form the geometry with rescaled
metric $g_{0,f}$ {as in \ref{may0801}}. Then
$$
\hcycl(\cV,\can(g_{0,f}))\in \hbK\oben{0}(\Spec(R))_{(1)}\ .
$$ 
\begin{lem}\label{jun2220}
We have 
$$
{h}(\hcycl(\cV,\can(g_{0,f})))= {-}\frac{1}{2{[K:\Q]}}\sum_{\sigma\in \Spec(R)(\C)}  
{\log}(f(\sigma))\ .
$$
\end{lem}
\begin{proof} Use \eqref{jun2210} and Lemma \ref{jun2211}. 
\end{proof}


If $(\cL,h^{\cL})\in \hPic(\Spec(R))$, then we have a canonical extended geometry $\can(h^{\cL})$ on
$\cL$ and can form 
$$
\hat c(\cL,h^{\cL}):=\hcycl(\cL,\can(h^{\cL}))\in \hbK\oben{0}(\Spec(R))_{(1)}\
.$$
{
\begin{theorem}\label{thmArith}
The map $\hat c\colon \hPic(\Spec(R)) \to \hbK\oben{0}(\Spec(R))_{(1)}$ is an isomorphism.
Furthermore, for any metrized line bundle $(\cL, h^\cL)$ we have
\[
\hdeg(\cL,h^{\cL})=  h(\hat c(\cL,h^{\cL}))\ .
\]
\end{theorem}
}
\begin{proof}
{
Since all connections involved are trivial, we have $\can(h^{\cL}\otimes h^{\cL'}) = \can(h^{\cL})
\otimes \can(h^{\cL'})$. Thus $\hat c$ is a group homomorphism.
}

{ 
There is a natural map $[\R^{\Spec(R)(\C)}]^{\Gal(\C/\R)} \to \hPic(\Spec(R))$ which sends the tuple
$\lambda=(\lambda_{\sigma})$ to the trivial line bundle $R$ with the metric $h^{(\lambda)}$ given by
$h^{(\lambda)}_{\sigma}(1) = \exp(-2\lambda_{\sigma})$. Recall that $H^{-1}(\IDR(\Spec(R)))   \cong
[\R^{\Spec(R)(\C)}]^{\Gal(\C/\R)}$.
We claim that we have a commutative diagram with exact rows
\[
\xymatrix{
0 \ar[r] & H^{-1}(\IDR(\Spec(R)))/\im(\beil) \ar[r] \ar@{=}[d] & \hPic(\Spec(R)) \ar[d]^{\hat c}
\ar[r] & \Pic(\Spec(R)) \ar[d]^{\cong} \ar[r] & 0\\
0 \ar[r] & H^{-1}(\IDR(\Spec(R)))/\im(\beil) \ar[r] & \hbK\oben{0}(\Spec(R))_{(1)} \ar[r] &
\bK^{0}(\Spec(R))_{(1)} \ar[r] & 0.
}
\]
Indeed, the right vertical map is given by the
topological cycle map, and it is known to be an
isomorphism. The exactness of the upper row is straightforward, the lower row is \eqref{jul0501}.
Finally, the commutativity of the left hand square follows from Lemma \ref{jun2211}. 
}

{In particular, $\hat c$ is an isomorphism.}

{ 
For the second assertion, it suffices
by the construction of $h$ to check that for $\lambda=(\lambda_{\sigma}) \in
[\R^{\Spec(R)(\C)}]^{\Gal(\C/\R)}$ we have
\[
\hdeg(R,h^{(\lambda)}) = \frac{1}{{[K:\Q]}}\sum_{\sigma\in \Spec(R)(\C)} \lambda_{\sigma}.
\]
But this is clear from the definition of $h^{(\lambda)}$ and  \eqref{aug0101} with $s=1$. 
}
\end{proof}

\section{Formality of the algebraic $K$-theory of number rings}\label{jun1004}

{
Let $M\R$ be the Moore spectrum of $\R$. For any spectrum $E$, we use the notation $E\R:= E\wedge
M\R$ for its realification.
}

Let $E\in \CAlg(\Sp^{\wedge})$ be a commutative ring spectrum.
Then we can form the differential graded commutative algebra  $ \pi_*(E\R)\in
\CAlg(\Ch^{\otimes})$ {with} trivial differentials. There is a unique
  equivalence class of maps 
$$
{r}\colon E\to H({\pi_*(E\R)})
$$ of spectra which induces the canonical realification map in homotopy. 
\begin{ddd}
The commutative ring spectrum  $E$ is called formal over $\R$ if ${r}$ can be refined to a 
morphism of
commutative
ring spectra.
\end{ddd}

If $\pi_*(E\R)$ is a free commutative $\R$-algebra, then $E$ is formal over $\R$ (see \cite{skript} for an argument). This applies e.g.~to complex bordism $\mathbf{MU}$ or connective complex $K$-theory $\mathbf{ku}$. From the formality of $\mathbf{ku}$ one can deduce the formality over $\R$ of
periodic complex $K$-theory $\mathbf{KU}$.

More generally, let $E\in \Fun(S, \CAlg(\Sp^{\wedge}))$ be a diagram of commutative
ring
spectra. It gives rise to a diagram $\pi_{*}(E\R)\in \Fun(S,\CAlg(\Ch^{\otimes}))$ of chain
complexes with trivial differential.
\begin{ddd}\label{fkwelfwefwefewfewfefe}
We say that $E$ is formal over $\R$ if there exists an equivalence
$E\R\simeq H(\pi_{*}(E\R))$ of diagrams of commutative ring spectra which {induces} the
identity on homotopy.
\end{ddd}

We let $\bS\subseteq \Reg_{\Z}$ be the full subcategory  whose objects are 
 spectra of rings of integers in number fields.
\begin{theorem}\label{juni1001}
The restriction of {the sheaf of algebraic $K$-theory spectra} $\bK$ to $\bS$ is formal over
$\R$.
\end{theorem}
\begin{proof}
We first show that the restriction of $H(\IDR)$ to $\bS$ is formal over $\R$. To this
end we describe, for every ring of integers $R$ in a number field $K$,  canonical 
representatives of the cohomology of $\IDR(\Spec(R))$. 
We have 
$$
\IDR(\Spec(R))(p) \cong \left(\{\omega\in {A(I)[2p]} \:|\: \omega_{|\{0\}} \in (2\pi i)^{p}\R\ , \quad \omega_{|\{1\}}=0 \}^{\Spec(R)(\C)}\right)^{\Gal(\C/\R)}
$$
for $p\ge 1$,
and 
$$
\IDR(\Spec(R))(0) \cong \left(\{\omega\in {A(I)} \:|\: \omega_{|\{0\}} \in  \R   \}^{\Spec(R)(\C)}\right)^{\Gal(\C/\R)}\ .
$$
We have 
$$
H^{*}\left(\left\{\omega\in {A(I)[2p]}\:\big|\: \omega_{|\{0\}} \in (2\pi i)^{p}\R\ , \quad
\omega_{|\{1\}}=0 \right\}\right)\cong 
{\begin{cases}
i^{p+1}\R, & *=-2p+1, \\ 
0, & \text{else,}
\end{cases}}
$$
and
$$
H^{*}\left(\left\{\omega\in A(I)\:\big|\: \omega_{|\{0\}} \in  \R    \right\}\right)\cong 
{\begin{cases}
\R, & *=0, \\ 
0, & \text{else.}
\end{cases}}
$$
Explicit representatives of generators are given by
 $i^{p+1}dt$ (with $t$ the coordinate of $I$) in the first case and $1$ in the second.
For real embeddings $\sigma \in \Spec(R)(\C)$ and odd $p\in \nat$, and for complex embeddings
$\sigma \in \Spec(R)(\C)$ and all $p\in \nat_{>0}$, we define the following elements in
$\IDR(\Spec(R))(p)$: For real $\sigma$,
\[
x(\sigma)_{1-2p}:=\left(\Spec(R)(\C)\ni \sigma^{\prime} \mapsto 
{\begin{cases}
i^{p+1}dt, &\sigma' =\sigma, \\
0, & \text{else,}
\end{cases}}
\right)\in \IDR(\Spec(R))(p) 
\]
and for complex $\sigma$  
\[
x(\sigma)_{1-2p}:=\left(\Spec(R)(\C)\ni \sigma^{\prime} \mapsto
{\begin{cases}
i^{p+1}dt, & \sigma' =\sigma, \\
(-1)^{p+1} i^{p+1} dt, & \sigma' = \bar\sigma, \\
0, & \text{else,}
\end{cases}}
\right)\in \IDR(\Spec(R))(p).   
\]
We let $M^{\prime}(R)\subseteq \IDR(\Spec(R))$ 
be the $\R$-submodule generated by elements $x(\sigma)_{1-2p}$ for $\sigma$ and $p$ as above. 

%
It is easy to see that the inclusion
$$
H^{*}(\IDR(\Spec(R)))\cong  \R\oplus M^{\prime}(R)\subset  \IDR(\Spec(R))
$$
is a quasi-isomorphism of commutative differential graded algebras which is natural in $R$. 
We therefore get a morphism of {diagrams of} ring spectra
$$
\beil\colon \bK_{|\bS} \to H(\IDR)_{|\bS}  \simeq H(H^{*}(\IDR)_{|\bS})\ .
$$
By Theorem \ref{may2760} the induced map 
\begin{equation}\label{eq:borel}
\pi_*(\bK_{|\bS})\otimes \R\to  H^{-*}(\IDR_{|\bS})
\end{equation}
coincides with Beilinson's regulator, which itself coincides up to a factor of $2$ with Borel's
regulator map \cite[Theorem~10.9]{burgos}. By Borel's results \cite{borel}, \eqref{eq:borel}
is injective, and the image is the kernel of the map
$$
p\colon \R\oplus M^{\prime}(R)\to \R\ , \quad b\mapsto \sum_{\sigma\in \Spec(R)(\C)}
n(\sigma)_{-1}(b)\ ,
$$
where the $n(\sigma)_{-1}(b)$ are the coefficients of $b$ in front of the generators $x(\sigma)_{-1}$. We define the  subspace $M(R):=\ker(p)\cap M^{\prime}(R) $.
Then we can define a canonical splitting 
$$
M^{\prime}(R)\to M(R)\ , \quad 
b\mapsto b-\frac{p(b)}{[{K}:\Q]}\sum_{\sigma\in \Spec(R)(\C)} x(\sigma)_{-1}\ .
$$
It induces a canonical ring homomorphism
$\R\oplus M^{\prime}(R)\to \R\oplus M(R)$ which is left-inverse to the inclusion $\R\oplus
M(R){\hookrightarrow} \R\oplus M^{\prime}(R)$ and therefore a map of diagrams of ring spectra
$s\colon H(\R\oplus M^{\prime})\to H(\R\oplus M)$ such that the composition
$$\bK\R_{|\bS} \xrightarrow{\beil\wedge M\R} H(\R\oplus M^{\prime})\stackrel{s}{\to} H(\R\oplus M) \simeq H(\pi_{*}(\bK_{|\bS}))$$
is an equivalence of diagrams of commutative ring spectra. 
\end{proof}

Observe that the structure of the homotopy groups of $\bK(\Spec(R))\R$ implies that all Massey
products are trivial. This can be considered as an $A_{\infty}$-version of formality. The additional
information given by Theorem \ref{juni1001} is that $\bK(\Spec(R))$
is formal in the commutative sense and in a way which is natural in the ring $R$.

\bibliographystyle{amsalpha}
\bibliography{zweite}
\end{document}